\documentclass[a4paper]{article}

\usepackage{amssymb}

\usepackage{amsmath,amsfonts,amssymb,amscd,bezier,amstext,
makeidx,amsthm,latexsym,amsmath,makeidx,tikz}
\usepackage[T1]{fontenc}
\usepackage[all]{xy}
\usepackage{color}
\usepackage{graphicx}
\usepackage{caption}
\usepackage{indentfirst}      
\usepackage{dsfont}
\usepackage{hyperref}
\usepackage[normalem]{ulem}

\newcommand{\R}{\mathbb R}
\newcommand{\Sn}{\mathbb{S}^n}
\newcommand{\Snc}{\mathbb{S}^n_c}
\newcommand{\Sk}{\mathbb{S}^k}
\newcommand{\Sp}{\mathbb{S}}
\newcommand{\spam}{\mathrm{span}}
\newcommand{\pie}{\langle}
\newcommand{\pid}{\rangle}
\newcommand{\Cin}{\text{C}^\infty(M)}
\newcommand{\tr}{\mathrm{trace}\,}
\newcommand{\hs}{\mathrm{Hess}\,}

\newcommand{\ddt}{\frac{\partial}{\partial t}}
\newcommand{\grad}{\mathrm{grad}\,}
\newcommand{\Div}{\mathrm{div}\,}
\newcommand{\til}{\widetilde}

\newtheorem{theorem}{Theorem}

\newtheorem{prop}[theorem]{Proposition}

\newtheorem{rem}[theorem]{Remark}

\theoremstyle{definition}\newtheorem{example}[theorem]{Example}

\title{Minimal immersions of Riemannian manifolds in
products of space forms}

\author{Fernando Manfio and Feliciano Vit\'orio}

\date{}

\begin{document}

\maketitle

\begin{abstract}
In this note, we give natural extensions to cylinders and tori of a classical result due to T. Takahashi \cite{takahashi} about minimal immersions into spheres. More precisely, we deal with Euclidean isometric immersions whose projections in $\mathbb{R}^N$ satisfy a spectral condition of their Laplacian.
\end{abstract}

\vspace{2ex}


\section{Introduction}

An isometric immersion $f:M^m\to N^n$ of a Riemannian manifold
$M$ in another Riemannian manifold $N$ is said to be {\em minimal}
if its mean curvature vector field $H$ vanishes. The study of minimal
surfaces is one of the oldest subjects in differential geometry, having
its origin with the work of Euler and Lagrange. In the last century, 
a series of works have been developed 
in the study of properties of minimal immersions, whose ambient 
space has constant sectional curvature. In particular, minimal
immersions in the sphere $\Sn$ play a important role in the theory,
as for example the famous paper of J. Simons \cite{simons}.

Let $f:M^m\to\R^n$ be an isometric immersion of a $m$-dimensional
manifold $M$ into the Euclidean space $\R^n$. Associated with
the induced metric on $M$, it is defined the Laplace operator
$\Delta$ acting on $\mathrm C^\infty(M)$. This Laplacian can
be extended in a natural way to the immersion $f$. A well-known
result by J. Eells and J. H. Sampson \cite{eells} asserts that the 
immersion $f$ is minimal if and only if $\Delta f=0$. The following
result, due to T. Takahashi \cite{takahashi}, states that the 
immersion $f$ {\em realizes} a minimal immersion in a sphere if and
only if its coordinate functions are eigenfunctions of the Laplace
operator with the same nonzero eigenvalue. 

\begin{theorem}\label{teo:takahashi}
Let $F:M^m\to\R^{n+1}$ be an isometric immersion such that
\[
\Delta F=-mcF
\]
for some constant $c\neq0$. Then $c>0$ and there exists a minimal 
isometric immersion $f:M^m\to\Snc$ such that $F=i\circ f$.
\end{theorem}


O. Garay generalized the Theorem \ref{teo:takahashi} for the
hypersurfaces $f:M^n\to\R^{n+1}$ satisfying $\Delta f=Af$, 
where $A$ is a constant $(n+1)\times(n+1)$ diagonal matrix.
He proved in \cite{garay} that such a hypersurface is either 
minimal or an open subset of a sphere or of a cylinder. In this
direction, J. Park \cite{park} classified the hypersurfaces in a
space form or in Lorentzian space whose immersion $f$
satisfies $\Delta f=Af+B$, where $A$ is a constant square
matrix and $B$ is a constant vector. Similar results were
obtained in \cite{alias}, where the authors study and classify
pseudo-Riemannian hypersurfaces in pseudo-Riemannian
space forms which satisfy the condition $\Delta f=Af+B$,
where $A$ is an endomorphism and $B$ is a constant vector.
In a somewhat different direction, B. Chen in a series of papers
discuss the problem of determining the geometrical structure 
of a submanifold knowing some simple analytic information
(see for example \cite{chen1}, \cite{chen2}).

In this work we shall deal with an isometric immersion
$f:M^m\to\R^N$ of a Riemannian manifold $M^m$ into the
Euclidean space $R^N$. If the submanifold $f(M)$ is contained 
in a cylinder $\Sn\times\R^k\subset\R^N$ or in a torus 
$\Sn\times\Sk\subset\R^N$, we shall call that the {\em immersion $f$
realizes an immersion in a cylinder} or in a {\em torus}, respectively.
Motivated by recent works on the submanifold theory in the
product of space forms \cite{tojvitorio}, we obtain theorems
that give us necessary and sufficient conditions for an isometric
immersion $f:M^m\to\R^N$ realize a minimal immersion in a
cylinder or in a torus (cf. Theorems \ref{teo:takacilindro} and 
\ref{teo:takatorus}).

\section{Preliminaries}

Let $M^m$ be a Riemannian manifold and $h\in\Cin$. The
{\em hessian} of $h$ is the symmetric section of 
$\text{Lin}(TM\times TM)$ defined by
\[
\hs h(X,Y)=XY(h)-\nabla_XY(h)
\]
for all $X,Y\in TM$. Equivalently,
\[
\hs h(X,Y)=\pie\nabla_X\grad h,Y\pid
\]
where $X,Y\in TM$ and $\grad h$ is the gradient of $h$. The {\em Laplacian} 
$\Delta h$ of a function $h\in\Cin$ at the point $p\in M$ is defined as
\[
\Delta h(p)=\tr\hs h(p)=\Div\grad h(p).
\]

Consider now an isometric immersion $f:M^m\to\R^n$. For a fixed 
$v\in\R^n$, let $h\in\Cin$ be the height function with respect to the
hyperplane normal to $v$, given by $h(p)=\pie f(p),v\pid$.
Then
\begin{eqnarray}\label{eq:hessalpha}
\hs h(X,Y)=\pie\alpha_f(X,Y),v\pid
\end{eqnarray}
for any $X,Y\in TM$. For an isometric immersion $f:M^n\to\R^n$,
by $\Delta f(p)$ at the point $p\in M$ we mean the vector 
\[
\Delta f(p)=(\Delta f_1(p),\ldots,\Delta f_n(p)),
\]
where $f=(f_1,\ldots,f_n)$. Taking traces in \eqref{eq:hessalpha} 
we obtain
\begin{eqnarray}\label{eq:laplacianH}
\Delta f(p)=mH(p),
\end{eqnarray}
where $H(p)$ is the mean curvature vector of $f$ at $p\in M$.

\section{Minimal submanifolds in $\Snc\times\R^k$}

Let $\Snc$ denote the sphere with constant sectional curvature
$c>0$ and dimension $n$. We use the fact that $\Snc$ admits a
canonical isometric embedding in $\R^{n+1}$ as
\[
\Snc=\left\{X\in\R^{n+1}:\pie X,X\pid=1/c\right\}.
\]
Thus, $\Snc\times\R^k$ admits a canonical isometric embedding 
\[
i:\Snc\times\R^k\to\R^{n+k+1}.
\]
Denote by $\pi:\R^{n+k+1}\to\R^{n+1}$ the canonical projection.
Then, the normal space of $i$ at each point $z\in\Snc\times\R^k$ is
spanned by $N(z)=c(\pi\circ i)(z))$, and the second fundamental form 
of $i$ is given by
\[
\alpha_i(X,Y)=-c\pie\pi X,Y\pid\pi\circ i.
\]
If we consider a parallel orthonormal frame $E_1,\ldots,E_{n+k+1}$
of $\R^{n+k+1}$ such that
\begin{eqnarray}\label{eq:frameSncRk}
\R^k=\spam\{E_{n+2},\ldots,E_{n+k+1}\},
\end{eqnarray}
we can express the second fundamental form $\alpha_i$ as
\begin{eqnarray}\label{eq:secondffi}
\alpha_i(X,Y)=-c\left(\pie X,Y\pid-\sum_{i=n+2}^{n+k+1}
\pie X,E_i\pid\pie Y,E_i\pid\right)\pi\circ i.
\end{eqnarray}

The following result shows that minimal immersions of a
$m$-dimensional Riemannian manifold into the cylinder 
$\Snc\times\R^k$ are precisely those immersions whose
$n+1$ first coordinate functions in $\R^{n+k+1}$ are
eigenfunctions of the Laplace operator in the induced
metric.

\begin{prop}\label{prop:SncRk}
Let $f:M^m\to\Snc\times\R^k$ be an isometric immersion and
\linebreak
set $F=i\circ f$, where $i:\Snc\times\R^k\to\R^{n+k+1}$ is the
canonical inclusion. Let $E_1,\ldots,E_{n+k+1}$ be a parallel
orthonormal frame of $\R^{n+k+1}$ as in \eqref{eq:frameSncRk}.
Then $f$ is a minimal immersion if and only if
\begin{eqnarray}\label{eq:Lapcilindro}
\Delta F=-c\left(m-\sum_{j=n+2}^{n+k+1}\|T_j\|^2\right)\pi\circ F,
\end{eqnarray}
where $T_j$ denotes the orthogonal projection of $E_j$ onto
$TM$.
\end{prop}
\begin{proof}
The second fundamental forms of $f$ and $F$ are related by
\[
\alpha_F(X,Y)=i_\ast\alpha_f(X,Y)+\alpha_i(f_\ast X,f_\ast Y)
\]
for all $X,Y\in TM$. From \eqref{eq:secondffi} we get that
\[
\alpha_F(X,Y)=i_\ast\alpha_f(X,Y)-c\left(\pie X,Y\pid-
\sum_{j=n+2}^{n+k+1}
\pie X,T_j\pid\pie Y,T_j\pid\right)\pi\circ F,
\]
where $T_j$ denotes the orthogonal projection of $E_i$ onto
$TM$. Taking traces and using \eqref{eq:laplacianH} yields 
\[
\Delta F=mi_\ast H^f-c\left(m-\sum_{j=n+2}^{n+k+1}\|T_j\|^2\right)
\pi\circ F,
\]
and the conclusion follows.
\end{proof}

\begin{rem}
In case $f:M^m\to\Snc\times\R$, a tangent vector field $T$ on $M$
and a normal vector field $\eta$ along $f$ are defined by
\[
\frac{\partial}{\partial t}=f_\ast T+\eta,
\]
where $\frac{\partial}{\partial t}$ is an unit vector field tangent to 
$\R$. In this case, $f$ is a minimal immersion if and only if
\[
\Delta F=-c(m-\|T\|^2)\pi\circ F.
\]
\end{rem}

The next result states that any isometric immersion of a
Riemannian manifold $M^m$ into Euclidean space $\R^{n+k+1}$,
whose Laplacian satisfies a condition as in \eqref{eq:Lapcilindro},
arises for a minimal isometric immersion of $M$ into some cylinder
$\Sn_c\times\R^k$.

\begin{theorem}\label{teo:takacilindro}
Let $F:M^m\to\R^{n+k+1}$ be an isometric immersion and let
$E_1,\ldots,E_{n+k+1}$ be a parallel orthonormal frame in 
$\R^{n+k+1}$ such that
\[
\Delta F=-c\left(m-\sum_{j=n+2}^{n+k+1}\|T_j\|^2\right)
\left(F-\sum_{j=n+2}^{n+k+1}\pie F,E_j\pid E_j\right),
\]
for some constant $c\neq0$, where $T_j$ denotes the orthogonal
projection of $E_j$ onto the tangent bundle $TM$. Then $c>0$ and 
there exists a minimal isometric immersion $f:M^m\to\Snc\times\R^k$ 
such that $F=i\circ f$.
\end{theorem}
\begin{proof}
Since $\Delta F=mH$ by \eqref{eq:laplacianH}, the assumption 
implies that the vector field 
\[
N=F-\sum_{j=n+2}^{n+k+1}\pie F,E_j\pid E_j
\]
is normal to $F$. On the other hand,
\begin{eqnarray*}
\pie N,E_j\pid &=& \left\pie F-\sum_{l=n+2}^{n+k+1}
\pie F,E_l\pid E_l,E_j\right\pid \\
&=& \pie F,E_j\pid - \pie F,E_j\pid=0
\end{eqnarray*}
for all $n+2\leq j\leq n+k+1$. Hence, for any $X\in TM$ we have
\[
X\pie N,N\pid=2\left\pie F_\ast X-\sum_{j=n+2}^{n+k+1}
\pie F_\ast X,E_j\pid E_j,N\right\pid=0,
\]
and it follows that $\pie N,N\pid=r^2$ for some constant $r$.
Now we claim that
\begin{eqnarray}\label{eq:LaplaceF}
\Delta\|F\|^2=2\sum_{j=n+2}^{n+k+1}\|T_j\|^2.
\end{eqnarray}
To see this, fix a point $p\in M$ and consider a local geodesic
frame $\{X_1,\ldots,X_m\}$ in $p$. Then
\[
\grad\|F\|^2=\sum_{\alpha=1}^mX_\alpha(\|F\|^2)X_\alpha=
2\sum_{\alpha=1}^m\pie F_\ast X_\alpha,F\pid X_\alpha=2F^T.
\]
Since $N$ is normal to $F$, we have
\[
F^T=\sum_{j=n+2}^{n+k+1}\pie F,E_j\pid T_j=
\sum_{j=n+2}^{n+k+1}\sum_{\alpha=1}^m
\pie F,E_j\pid\pie E_j,X_\alpha\pid X_\alpha,
\]
and it follows that
\[
\grad\|F\|^2=2\sum_{j=n+2}^{n+k+1}\sum_{\alpha=1}^m
\pie F,E_j\pid\pie E_j,X_\alpha\pid X_\alpha.
\]
Therefore,
\begin{eqnarray*}
\Delta\|F\|^2 &=& \sum_{\beta=1}^m\left\pie\nabla_{X_\beta}\grad\|F\|^2,
X_\beta\right\pid \\
&=& 2\sum_{\alpha,\beta=1}^m\sum_{j=n+2}^{n+k+1}\left\pie\nabla_{X_\beta}
\pie F,E_j\pid\pie E_j,X_\alpha\pid X_\alpha,X_\beta\right\pid \\
&=& 2\sum_{\alpha,\beta=1}^m\sum_{j=n+2}^{n+k+1}\pie F_\ast X_\beta,E_j\pid
\pie E_j,X_\alpha\pid\pie X_\alpha,X_\beta\pid \\
&=& 2\sum_{\alpha=1}^m\sum_{j=n+2}^{n+k+1}\pie X_\alpha,T_j\pid^2 =
2\sum_{j=n+2}^{n+k+1}\|T_j\|^2,
\end{eqnarray*}
and this proves our claim. Finally, using the fact that
\[
\Delta\|F\|^2=2(\pie\Delta F,F\pid+m),
\]
we get that 
\begin{eqnarray*}
\sum_{j=n+2}^{n+k+1}\|T_j\|^2 &=& \pie\Delta F,F\pid+m=
\left\pie-c\left(m-\sum_{j=n+2}^{n+k+1}\|T_j\|^2\right)N,N\right\pid+m \\
&=& -\left(m-\sum_{j=n+2}^{n+k+1}\|T_j\|^2\right)cr^2+m,
\end{eqnarray*}
and the equality above implies that $c=1/r^2$. We conclude that there
exists an isometric immersion $f:M^m\to\Snc\times\R^k$ such that
$F=i\circ f$, and minimality of $f$ follows from Proposition
\ref{prop:SncRk}.
\end{proof}

As an application of Theorem \ref{teo:takacilindro} we will
construct an example into $\Sp^{2n-1}\times\R$.

\begin{example}\label{exa:clifford}
For each real number $a$, with $\sqrt{n-1}<a\leq\sqrt n$,
we claim that there exists a real number $b$ such that the 
immersion $f:\R^n\to\R^{2n+1}$, given by
\[
f(x_1,\ldots,x_n)=\frac{1}{\sqrt{n}}\left( e^{i ax_1},
\ldots,e^{iax_n},b\sum_{j=1}^nx_j\right),
\]
is a minimal immersion into $\Sp^{2n-1}\times\R$. In fact, 
we need check the hypothesis of the Theorem
\ref{teo:takacilindro} for a suitable choice of $(a,b)\in\R^2$.
First observe that 
\[
\Delta f = - a^2\left(f-\left\pie f,\ddt\right\pid\ddt\right),
\]
where $\frac{\partial}{\partial t}$ denotes a unit vector field
tangent to the factor $\R$. Now, if we denote by $T$ the 
orthogonal projection of $\frac{\partial}{\partial t}$ onto $f$,
a direct computation give us 
\begin{eqnarray}\label{eq:example1}
\|T\|^2= \frac{nb^2}{a^2+nb^2}.
\end{eqnarray}
On the other hand, follows from \eqref{eq:LaplaceF} that
\begin{eqnarray}\label{eq:example2}
\|T\|^2=\pie\Delta f,f\pid+n=n-a^2.
\end{eqnarray}
It follows from \eqref{eq:example1} and \eqref{eq:example2} that
\[
b^2=\frac{a^2(n-a^2)}{n(a^2-n+1)}.
\]
\end{example}

\section{Minimal submanifolds in the product $\Sn\times\Sk$}

Let $\Sn$ and $\Sk$ denote the spheres of dimension $n$ 
and $k$, respectively. Using the fact that the spheres
admit a canonical isometric embedding $\Sn\subset\R^{n+1}$
and $\Sk\subset\R^{k+1}$, the product $\Sn\times\Sk$ admits
a canonical isometric embedding 
\begin{eqnarray}\label{eq:inclusioncan}
i:\Sn\times\Sk\to\R^{n+k+2}.
\end{eqnarray}
Denote by $\pi_1:\R^{n+k+2}\to\R^{n+1}$ and 
$\pi_2:\R^{n+k+2}\to\R^{k+1}$ the canonical projections.
Then, the normal space of $i$ at each point 
$z\in\Sn\times\Sk$ is spanned by $N_1(z)=\pi_1(i(z))$
and $N_2(z)=\pi_2(i(z))$, and the second fundamental
form of $i$ is given by
\[
\alpha_i(X,Y)=-\pie\pi_1X,Y\pid N_1-\pie\pi_2X,Y\pid N_2
\]
for all $X,Y\in T_z\Sn\times\Sk$.   

\vspace{.2cm}

Now, let $f:M^m\to\Sn\times\Sk$ be an isometric immersion
of a Riemannian manifold. Then, writing $F=i\circ f$, the unit
vector fields $N_1=\pi_1\circ F$ and $N_2=\pi_2\circ F$ are
normal to $F$. Consider a parallel orthonormal frame
$E_1,\ldots,E_{n+k+2}$ of $\R^{n+k+2}$ such that
\begin{eqnarray}\label{eq:fieldnormalF}
\R^{n+1}=\spam\{E_1,\ldots,E_{n+1}\}\quad\text{and}\quad
\R^{k+1}=\spam\{E_{n+2},\ldots,E_{n+k+2}\}.
\end{eqnarray}
In terms of this frame, we can express the vector fields $N_1$
and $N_2$ as
\begin{eqnarray}\label{eq:fieldsN1N2}
N_1=F-\sum_{j=n+2}^{n+k+2}\pie F,E_j\pid E_j
\quad\text{and}\quad
N_2=F-\sum_{l=1}^{n+1}\pie F,E_l\pid E_l.
\end{eqnarray}

\begin{prop}\label{prop:SnSk}
Let $f:M^m\to\Sn\times\Sk$ be an isometric immersion and
set $F=i\circ f$, where $i:\Sn\times\Sk\to\R^{n+k+2}$ is the
canonical inclusion. Let $E_1,\ldots,E_{n+k+2}$ be a parallel
orthonormal frame of $\R^{n+k+2}$ as in \eqref{eq:fieldnormalF}.
Then $f$ is a minimal immersion if and only if
\begin{eqnarray}\label{eq:Laplacegeral}
\Delta F=-\left(m-\sum_{j=n+2}^{n+k+2}\|T_j\|^2\right)N_1
-\left(m-\sum_{l=1}^{n+1}\|T_l\|^2\right)N_2,
\end{eqnarray}
where $T_j$ denotes the orthogonal projection of $E_j$ onto
$TM$.
\end{prop}
\begin{proof}
The second fundamental forms of $f$ and $F$ are related by
\[
\alpha_F(X,Y)=i_\ast\alpha_f(X,Y)+\alpha_i(f_\ast X,f_\ast Y)
\]
for all $X,Y\in TM$. Let $E_1,\ldots,E_{n+k+2}$ be a parallel
orthonormal frame of $\R^{n+k+2}$ as in \eqref{eq:fieldnormalF}.
Given $X\in TM$, we can write
\[
\pi_1X=X-\sum_{j=n+1}^{n+k+2}\pie X,T_j\pid E_j
\quad\text{and}\quad
\pi_2X=X-\sum_{l=1}^{n+1}\pie X,T_l\pid E_l,
\]
and so, we have
\[
\pie\pi_1X,Y\pid=\pie X,Y\pid-\sum_{j=n+1}^{n+k+2}
\pie X,T_j\pid\pie Y,T_j\pid
\]
and
\[
\pie\pi_2X,Y\pid=\pie X,Y\pid-\sum_{l=1}^{n+1}
\pie X,T_l\pid\pie Y,T_l\pid.
\]
Then the second fundamental form of $F$ can be
expressed by
\begin{eqnarray*}
\alpha_F(X,Y) &=& i_\ast\alpha_f(X,Y)-\left(\pie X,Y\pid-
\sum_{j=n+1}^{n+k+2}\pie X,T_j\pid\pie Y,T_j\pid\right)N_1 \\
&&-\left(\pie X,Y\pid-\sum_{l=1}^{n+1}
\pie X,T_l\pid\pie Y,T_l\pid\right)N_2.
\end{eqnarray*}
Taking traces and using \eqref{eq:laplacianH} yields
\[
\Delta F=mH^F=mi_\ast H^f-
\left(m-\sum_{j=n+1}^{n+k+2}\|T_j\|^2\right)N_1
-\left(m-\sum_{l=1}^{n+1}\|T_l\|^2\right)N_2
\]
and the conclusion follows.
\end{proof}

\begin{rem}
Observe that an isometric immersion $f:M^m\to\Sn\times\Sk$
can be seen as an isometric immersion 
$\til f=\iota\circ f:M^m\to\Sp_\kappa^{n+k+1}$ 
into the sphere with constant sectional curvature $\kappa=1/2$,
where $\iota:\Sn\times\Sk\to\Sp_\kappa^{n+k+1}$ denotes
the canonical inclusion.
\end{rem}

The next result states that any isometric immersion of a
Riemannian manifold $M^m$ into the sphere $\Sp_\kappa^{N-1}$
with constant sectional curvature $\kappa=1/2$, whose 
Laplacian of coordinate functions satisfies a condition as in 
\eqref{eq:Laplacegeral}, arises for a minimal isometric immersion 
of $M^m$ into a product of spheres $\Sn\times\Sk\subset\R^N$.

\begin{theorem}\label{teo:takatorus}
Let $F:M^m\to\Sp_\kappa^{N-1}$ be an isometric immersion.
Fixed a choice of two integers $n$ and $k$, with $N=n+k+2$,
let $E_1,\ldots,E_N$ be a parallel orthonormal frame in $\R^N$ 
as in \eqref{eq:fieldnormalF} such that
\[
\Delta\til F=-\left(m-\sum_{j=n+2}^{n+k+2}\|T_j\|^2\right)N_1
-\left(m-\sum_{l=1}^{n+1}\|T_l\|^2\right)N_2,
\]
where $\til F=h\circ F$, $h:\Sp_\kappa^{N-1}\to\R^N$ is the
umbilical inclusion, $T_i$ denotes
the orthogonal projection of $E_i$ onto $TM$ and $N_1$ and
$N_2$ as in \eqref{eq:fieldsN1N2}. Then 
there exists a minimal isometric immersion $f:M^m\to\Sn\times\Sk$
such that $F=i\circ f$.
\end{theorem}
\begin{proof}
We first prove that $N_1$ and $N_2$ are normal to $F$. In fact, 
in terms of an orthonormal frame $\{X_1,\ldots,X_m\}$ of $TM$, 
we have
\begin{eqnarray}\label{eq:sumTi}
\sum_{i=1}^N\|T_i\|^2=\sum_{l=1}^{n+1}\|T_l\|^2 +
\sum_{j=n+2}^N\|T_j\|^2=m.
\end{eqnarray}
Then, as $\til F=N_1+N_2$, we can write:
\begin{eqnarray}\label{eq:LaplacetilF}
\begin{aligned}
\Delta\til F =& -\left(\sum_{l=1}^{n+1}\|T_l\|^2\right)N_1 -
\left(\sum_{j=n+2}^N\|T_j\|^2\right)N_2 \\
=& -\left(\sum_{l=1}^{n+1}\|T_l\|^2\right)\til F +
\left(\sum_{l=1}^{n+1}\|T_l\|^2-\sum_{j=n+2}^N\|T_j\|^2\right)N_2.
\end{aligned}
\end{eqnarray}
If 
\[
\sum_{l=1}^{n+1}\|T_l\|^2=\sum_{j=n+2}^N\|T_j\|^2,
\]
we have, by using \eqref{eq:sumTi}, that
\[
\Delta\til F=-\frac{1}{2}m\til F.
\]
Thus, it follows from Theorem \ref{teo:takahashi} that 
$F:M^m\to\Sp_\kappa^{N-1}$ is a minimal isometric immersion.
Suppose from now on that 
\begin{eqnarray}\label{eq:difterms}
\sum_{l=1}^{n+1}\|T_l\|^2\neq\sum_{j=n+2}^N\|T_j\|^2.
\end{eqnarray}
As $\Delta\til F=mH$ and $\til F$ is normal to $M$, we conclude
from \eqref{eq:LaplacetilF} that $N_2$ is normal to $M$. Similary 
we obtain that $N_1$ is normal to $M$. Now, for $n+2\leq j\leq n+k+2$, 
we have
\begin{eqnarray*}
\pie N_1,E_j\pid &=& \left\pie\til F-\sum_{i=n+2}^N
\pie\til F,E_i\pid E_i,E_j\right\pid \\
&=& \pie\til F,E_j\pid-\pie\til F,E_j\pid=0.
\end{eqnarray*}
Hence, for any $X\in TM$ we have
\[
X\pie N_1,N_1\pid=2\left\pie\til F_\ast X-\sum_{j=n+2}^N
\pie\til F_\ast X,E_j\pid E_j,N_1\right\pid=0,
\]
and it follows that $\pie N_1,N_1\pid=r^2$ for some constant $r$.
The same argument gives $\pie N_2,N_2\pid=s^2$ for some
constant $s$. Since $\til F=N_1+N_2$ and 
$\Delta\|\til F\|^2=2(\pie\Delta\til F,\til F\pid+m)$, we have
\begin{eqnarray*}
0 &=& \frac{1}{2}\Delta\|\til F\|^2=\pie\Delta\til F,\til F\pid+m \\
&=& -\left(m-\sum_{j=n+2}^N\|T_j\|^2\right)r^2
-\left(m-\sum_{l=1}^{n+1}\|T_l\|^2\right)s^2+m.  \\
&=& -\left(\sum_{l=1}^{n+1}\|T_l\|^2\right)r^2 - 
\left(\sum_{j=n+2}^N\|T_j\|^2\right)s^2+m.
\end{eqnarray*}
Since $r^2+s^2=2$, we can rewrite the above equation as
\begin{eqnarray*}
\left(\sum_{l=1}^{n+1}\|T_l\|^2-\sum_{j=n+2}^N\|T_j\|^2\right)s^2
&=& 2\sum_{l=1}^{n+1}\|T_l\|^2-m \\
&=& \sum_{l=1}^{n+1}\|T_l\|^2-\sum_{j=n+2}^{n+k+2}\|T_j\|^2.
\end{eqnarray*}
As we are assuming \eqref{eq:difterms}, 
we obtain $s^2=1$, and therefore $r^2=1$. We conclude that 
there exists an isometric immersion $f:M^m\to\Sn\times\Sk$ such that
$F=i\circ f$, and minimality of $f$ follows from Proposition
\ref{prop:SnSk}.

\end{proof}

\vspace{1cm}

\begin{small}
\begin{tabular}{lcr}

\begin{tabular}{l}
Fernando Manfio\\
ICMC, Universidade de S\~ao Paulo\\
S\~ao Carlos-SP, 13561-060 \\
Brazil\\
\verb+manfio@icmc.usp.br+\\
\end{tabular}\
&\quad&
\begin{tabular}{l}
Feliciano Vit\'orio\\
IM, Universidade Federal de Alagoas\\
Macei\'o-AL, 57072-900 \\
Brazil\\
\verb+feliciano@pos.mat.ufal.br+\\
\end{tabular}
\end{tabular}
\end{small}

\end{document}